\documentclass[11pt,reqno]{amsart}
\usepackage{amsmath, amsthm, amscd, amsfonts, amssymb, graphicx, color}
\usepackage[bookmarksnumbered, colorlinks, plainpages]{hyperref}

\newtheorem{theorem}{Theorem}[section]
\newtheorem{lemma}[theorem]{Lemma}

\newtheorem{corollary}[theorem]{Corollary}
\theoremstyle{definition}

\theoremstyle{remark}

\numberwithin{equation}{section}

\newcommand{\be}{\begin{equation}}
\newcommand{\ee}{\end{equation}}
\newcommand{\bea}{\begin{eqnarray}}
\newcommand{\eea}{\end{eqnarray}}
\begin{document}
\begin{center}
{\large{\textbf{Almost Kenmotsu manifolds admitting certain vector fields}}}
\end{center}
\vspace{0.1 cm}

\begin{center}

Dibakar Dey and Pradip Majhi\\
Department of Pure Mathematics,\\
University of Calcutta,
35 Ballygunge Circular Road,\\
 Kolkata - 700019, West Bengal, India,\\
E-mail: deydibakar3@gmail.com; mpradipmajhi@gmail.com\\
\end{center}

\vspace{0.3 cm}
\textbf{Abstract:} In the present paper, we characterize almost Kenmotsu manifolds admitting holomorphically planar conformal vector (HPCV) fields. We have shown that if an almost Kenmotsu manifold $M^{2n+1}$ admits a non-zero HPCV field $V$ such that $\phi V = 0$, then $M^{2n+1}$ is locally a warped product of an almost Kaehler manifold and an open interval. As a corollary of this we obtain few classifications of an almost Kenmotsu manifold to be a Kenmotsu manifold and also prove that the integral manifolds of $\mathcal{D}$ are totally umbilical submanifolds of $M^{2n+1}$. Further, we prove that if an almost Kenmotsu manifold with positive constant $\xi$-sectional curvature admits a non-zero HPCV field $V$, then either $M^{2n+1}$ is locally a warped product of an almost Kaehler manifold and an open interval or isometric to a sphere. Moreover, a $(k,\mu)'$-almost Kenmotsu manifold admitting a HPCV field $V$ such that $\phi V \neq 0$ is either locally isometric to $\mathbb{H}^{n+1}(-4)$ $\times$ $\mathbb{R}^n$ or $V$ is an eigenvector of $h'$. Finally, an example is presented.\\

\textbf{Mathematics Subject Classification 2010:} 53D15,
 53C25.\\

\textbf{Keywords:} Almost Kenmotsu manifold, Holomorphically Planar conformal vector field, Almost Kaehler manifold, Totally umbilical submanifolds.\\

\section{ \textbf{Introduction}}
\vspace{0.3 cm}
In the present time, the study of existence of Killing vector fields in Riemannian manifolds is a very interesting topic as they preserves a given metric and determine the degree of symmetry of the manifold. Conformal vector fields whose flow preserves a conformal class of metrics are very important in the study of several kind of almost contact metric manifolds. \\

A smooth vector field $V$ on a Riemannian manifold $(M,g)$ is said to be conformal vector field if there exist a smooth function $f$ on $M$ such that
\bea
\pounds_V g = 2fg, \label{1.1}
\eea
 where $\pounds_V g$ is the Lie derivative of $g$ with respect to $V$. The vector field $V$ is called homothetic or Killing accordingly as $f$ is constant or zero. Moreover, $V$ is said to be closed conformal vector field if the metrically equivalent $1$-form of $V$ is closed. If the conformal vector field $V$ is gradient of some smooth function $\lambda$, then $V$ is called gradient conformal vector field. The geometry of conformal vector fields have been investigated in (\cite{deshmukh}, \cite{deshmukh6}).

  A vector field $V$ on a contact metric manifold $M^{2n+1}(\phi,\xi,\eta,g)$ is said to be holomorphically planar conformal vector field if it satisfies
 \bea
 \nabla_X V = aX + b\phi X \label{1.2}
 \eea
 for any vector field $X$, where $a$, $b$ are smooth functions on $M$. As a generalization of closed conformal vector fields Sharma \cite{sharma} introduced the notion of holomorphically planar conformal vector ( in short, HPCV ) fields on almost Hermitian manifold. In \cite{ghosh1}, Ghosh and Sharma characterize an almost Hermitian manifolds admitting a HPCV field. They shows that if $V$ is strictly non-geodesic non-vanishing HPCV field on an almost Hermitian manifold, then $V$ is homothetic and almost analytic. Further Sharma \cite{sharma1} shows that among all complete and simply connected $K$-contact manifolds only the unit sphere admits a non-Killing HPCV field and a $(k,\mu)$-contact manifold admitting a non-zero HPCV field is either Sasakian or locally isometric to $E^3$ or $E^{n+1} \times S^n(4)$. In \cite{ghosh}, Ghosh studied HPCV fields in the framework of contact metric manifolds under certain conditions and proved that a contact metric manifold with pointwise constant $\xi$-sectional curvature admitting a non-closed HPCV field $V$ is either $K$-contact or $V$ is homothetic.\\

 Motivated by the above studies we consider HPCV fields in the framework of a special type of almost contact metric manifolds, called almost Kenmotsu manifolds. The paper is organized as follows : 
 
 In section 2, we present some preliminary notions on almost Kenmotsu manifolds existing in the literature. Section 3 deals with HPCV fields on almost Kenmotsu manidolds and section 4 is associated to the study of HPCV fields on $(k,\mu)'$-almost Kenmotsu manifolds.

\section{\textbf{Preliminaries}}
\vspace{0.3 cm}
An almost contact structure on a $(2n+1)$-dimensional smooth manifold $M^{2n+1}$ is a triplet $(\phi, \xi, \eta)$, where $\phi$ is a $(1,1)$-tensor, $\xi$ is a global vector field  and $\eta$ is a $1$-form satisfying (\cite{bl}, \cite{bll}),
 \be
\phi^{2}=-I+\eta\otimes\xi,\;\; \eta(\xi)=1, \label{2.1}
\ee
where $I$ denote the identity endomorphism. Here also $\phi\xi=0$ and $\eta\circ\phi=0$ hold; both can be derived from (\ref{2.1}) easily. \\
If a manifold $M$ with a $(\phi, \xi, \eta)$-structure admits a Riemannian metric $g$ such that
\be
\nonumber g(\phi X,\phi Y)=g(X,Y)-\eta(X)\eta(Y),
\ee
for any vector fields $X$, $Y$ on $M^{2n+1}$, then $M^{2n+1}$ is said to be an almost contact metric manifold. The fundamental 2-form $\Phi$ on an almost contact metric manifold is defined by $$\Phi(X,Y)=g(X,\phi Y)$$ for any vector fields $X$, $Y$ on $M^{2n+1}$. The condition for an almost contact metric manifold being normal is equivalent to vanishing of the $(1,2)$-type torsion tensor $N_{\phi}$, defined by $$N_{\phi}=[\phi,\phi]+2d\eta\otimes\xi,$$ where $[\phi,\phi]$ is the Nijenhuis tensor of $\phi$ \cite{bl}. Recently in (\cite{dp}, \cite{dp1}, \cite{dp2}), almost contact metric manifold such that $\eta$ is closed and $d\Phi=2\eta\wedge\Phi$ are studied and they are called almost Kenmotsu manifolds. For more details on almost Kenmotsu manifolds we refer the reader to go through the references (\cite{ucd}, \cite{dp1}, \cite{dp2}). Obviously, a normal almost Kenmotsu manifold is a Kenmotsu manifold. Also Kenmotsu manifolds can be characterized by 
\bea
\nonumber (\nabla_{X}\phi)Y=g(\phi X,Y)\xi-\eta(Y)\phi X,
\eea
 for any vector fields $X$, $Y$. Let the distribution orthogonal to $\xi$ is denoted by $\mathcal{D}$, then $\mathcal{D} = Im(\phi) = Ker(\eta)$. Since $\eta$ is closed, $\mathcal{D}$ is an integrable distribution.\\
 
 The study of nullity distributions is a very interesting topic on almost contact metric manifolds. The notion of $k$-nullity distribution was introduced by Gray \cite{gr} and Tanno \cite{ta} in the study of Riemannian manifolds. Blair, Koufogiorgos and Papantonio \cite{bkp} introduced the generalized notion of the $k$-nullity distribution, named the $(k,\mu)$-nullity distribution on a contact metric manifold.  In \cite{dp}, Dileo and Pastore introduce the notion of $(k,\mu)'$-nullity distribution, another generalized notion of the $k$-nullity distribution, on an almost  Kenmotsu manifold  $(M^{2n+1}, \phi, \xi, \eta, g)$, which is defined for any $p \in M^{2n+1}$ and $k,\mu \in \mathbb {R}$ as follows:
\bea
 N_{p}(k,\mu)'=\{Z\in T_{p}M:R(X,Y)Z&=&k[g(Y,Z)X-g(X,Z)Y]\nonumber\\&&+\mu[g(Y,Z)h'X-g(X,Z)h'Y]\},\label{2s1.7}
\eea
where $h'=h\circ\phi$.\\

 Let $M^{2n+1}$ be an almost Kenmotsu manifold with structure $(\phi,\xi,\eta,g)$. The Levi-Civita connection satisfies $\nabla_\xi \xi = 0$ and $\nabla_\xi \phi = 0$. We denote by $h=\frac{1}{2}\pounds_{\xi}\phi$ and $l=R(\cdot, \xi)\xi$ on $M^{2n+1}$. The tensor fields $l$ and $h$ are symmetric operators and satisfy the following relations \cite{pv}:
 \bea
h\xi=0,\;l\xi=0,\;tr(h)=0,\;tr(h\phi)=0,\;h\phi+\phi h=0, \label{2.2}
 \eea
 We also have the following formulas given in (\cite{dp} - \cite{dp2})
\bea
\nabla_{X}\xi=X - \eta(X)\xi - \phi hX, \label{2.3}
\eea
\be
R(X,Y)\xi = \eta(X)(Y - \phi hY) - \eta(Y)(X - \phi hX) + (\nabla_Y \phi h)X - (\nabla_X \phi h)Y, \label{2.4}
\ee
\bea
(\nabla_X \phi)Y - (\nabla_{\phi X} \phi)\phi Y = -\eta(Y)\phi X - 2g(X,\phi Y)\xi - \eta(Y)hX \label{2.5}
\eea
for any $X,\;Y$ on $M^{2n+1}$.
The $(1,1)$-type symmetric tensor field $h'=h\circ\phi$ is anticommuting
with $\phi$ and $h'\xi=0$. Also it is clear that (\cite{dp}, \cite{waa})\be\label{2.7} h=0\Leftrightarrow
h'=0,\;\;h'^{2}=(k+1)\phi^2(\Leftrightarrow h^{2}=(k+1)\phi^2). \ee

\section{\textbf{HPCV fields on almost Kenmotsu manifolds}}
In this section we characterize almost Kenmotsu manifolds admitting a holomorphically planar conformal vector field $V$. Before proving our main theorems we first state and prove the following lemma.

\begin{lemma} \label{lem3.1}
Let $M^{2n+1}$ be an almost Kenmotsu manifold admitting a HPCV field $V$. Then the following relation
\bea
\nonumber \phi Va = 4nb\eta(V) + (\xi b)\eta(V) - Vb
\eea
holds on $M^{2n+1}$.
\end{lemma}
\textbf{Proof:} Differentiating (\ref{1.2}) covariantly along any vector field $Y$, we have
\bea
\nabla_Y \nabla_X V = a(\nabla_Y X) + (Ya)X + b(\nabla_Y \phi X) + (Yb)\phi X. \label{3.1}
\eea
Interchanging $X$ and $Y$ in the above equation, we get
\bea
\nabla_X \nabla_Y V = a(\nabla_X Y) + (Xa)Y + b(\nabla_X \phi Y) + (Xb)\phi Y. \label{3.2}
\eea
Replacing $X$ by $[X,Y]$ in (\ref{1.2}) yields
\bea
\nabla_{[X,Y]} V = a(\nabla_X Y) - a(\nabla_Y X) + b\phi(\nabla_X Y) - b\phi(\nabla_Y X). \label{3.3}
\eea
Now using $R(X,Y)Z = [\nabla_X,\nabla_Y]Z - \nabla_{[X,Y]}Z$ gives
\bea
\nonumber R(X,Y)V &=& (Xa)Y - (Ya)X + (Xb)\phi Y - (Yb)\phi X \\ && + b[(\nabla_X \phi)Y - (\nabla_Y \phi)X]. \label{3.4}
\eea
Putting $X = \phi X$ and $Y = \phi Y$ in (\ref{3.4}) we get
\bea
\nonumber R(\phi X,\phi Y)V &=& (\phi Xa)\phi Y - (\phi Ya)\phi X + (\phi Xb)[- Y + \eta(Y)\xi] \\ &&- (\phi Yb)[- X + \eta(X)\xi] + b[(\nabla_{\phi X} \phi)\phi Y - (\nabla_{\phi Y} \phi)\phi X]. \label{3.5}
\eea
Now adding equations (\ref{3.4}) and (\ref{3.5}) and using (\ref{2.5}) we have
\bea
\nonumber R(X,Y)V + R(\phi X,\phi Y)V &=& (Xa)Y - (Ya)X + (Xb)\phi Y - (Yb)\phi X \\ \nonumber && + (\phi Xa)\phi Y - (\phi Ya)\phi X - (\phi Xb)Y \\ \nonumber &&+ (\phi Xb)\eta(Y)\xi + (\phi Yb)X - (\phi Yb)\eta(X)\xi \\ \nonumber &&+ b[-\eta(Y)\phi X - 2g(X,\phi Y)\xi - \eta(Y)hX \\ && + \eta(X)\phi Y + 2g(\phi X,Y)\xi + \eta(X)hY]. \label{3.6}
\eea
Taking inner product of (\ref{3.6}) with $V$ and then substituting  $X = \phi X$ and $Y = \phi Y$ yields
\bea
\nonumber && (\phi Xa)g(\phi Y,V) - (\phi Ya)g(\phi X,V) + (\phi Xb)[- g(Y,V) + \eta(Y)\eta(V)] \\ \nonumber &&- (\phi Yb)[- g(X,V) + \eta(X)\eta(V)] + [- X + \eta(X)\xi](a)[- g(Y,V) + \eta(Y)\eta(V)]\\ \nonumber && - [- Y + \eta(Y)\xi](a)[- g(X,V) + \eta(X)\eta(V)] - [- X + \eta(X)\xi](b)g(\phi Y,V) \\ && + [- Y + \eta(Y)\xi](b)g(\phi X,V) - 4bg(X,\phi Y)\eta(V) = 0. \label{3.7}
\eea
Now replacing $Y$ by $\phi Y$ in the foregoing equation we obtain
\bea
\nonumber && - g(\phi Da,X)[- g(Y,V) + \eta(Y)\eta(V)] + g(Da,Y)g(\phi X,V) - \eta(Y)(\xi a)g(\phi X,V) \\ \nonumber && + g(\phi Db,X)g(Y,V) + g(Db,Y)[- g(X,V) + \eta(X)\eta(V)] -\eta(Y)(\xi b)[- g(X,V) \\ \nonumber && + \eta(X)\eta(V)] + g(Da,X)g(Y,V) - \eta(X)(\xi a)g(Y,V) + g(Da,Y)[- g(X,V) \\ \nonumber && + \eta(X)\eta(V)] + g(Db,X)[- g(Y,V) + \eta(Y)\eta(V)] - \eta(X)(\xi b)[- g(Y,V) \\ && + \eta(Y)\eta(V)] - g(Db,Y)g(\phi X,V) + 4bg(X,Y)\eta(V) - 4b\eta(X)\eta(Y)\eta(V) = 0. \label{3.8}
\eea
Contracting $X$ and $Y$ in (\ref{3.8}) we have
\bea
\nonumber - 2\phi Va - 2Vb + 2(\xi b)\eta(V) + 8nb\eta(V) = 0,
\eea
which implies
\bea
 \phi Va = 4nb\eta(V) + (\xi b)\eta(V) - Vb. \label{3.9}
\eea
This completes the proof.\\

\begin{theorem} \label{th1.1}
If an almost Kenmotsu manifold $M^{2n+1}$ admits a non-zero HPCV field $V$ such that $\phi V = 0$, then $M^{2n+1}$ is locally a warped product of an almost Kaehler manifold and an open interval.
\end{theorem}

\begin{proof}
Let $M^{2n+1}$ be an almost Kenmotsu manifold admitting a non-zero HPCV field $V$ such that $\phi V = 0$. Operating $\phi$ on it we get
\bea
V = \eta(V)\xi. \label{3.10}
\eea
Now using (\ref{3.10}) and $\phi V = 0$ in Lemma \ref{lem3.1} we have $4nb\eta(V) = 0$, which implies either $b = 0$ or $\eta(V) = 0$. If $\eta(V) = 0$, then from (\ref{3.10}) we have $V = 0$, which is a contradiction to our hypoyhesis. Thus we get $b = 0$.\\
 Differentiating (\ref{3.10}) covariantly along any vector field $X$ and using $b = 0$, $\phi V = 0$, (\ref{1.2}) and (\ref{2.3}) we obtain
 \bea
 aX = a\eta(X)\xi + g(X,V)\xi - 2\eta(X)\eta(V)\xi + \eta(V)X - \eta(V)\phi hX. \label{3.11}
 \eea
Contracting $X$ and using (\ref{2.2}) in (\ref{3.11}) yields $a = \eta(V)$. Substituting the value of $a$ in (\ref{3.11}) we get
\bea
g(X,V)\xi - \eta(X)\eta(V)\xi - \eta(V)\phi hX = 0. \label{3.12}
\eea
Replacing $X$ by $\phi X$ in the above equation and using the hypothesis $\phi V = 0$ and $\eta(V) \neq 0$ we infer that $hX = 0$ for any vector field $X$ on $M^{2n+1}$. The rest of the proof follows from Theorem 2 of \cite{dp2}.
\end{proof}

Proposition 1 of \cite{dp2} says that " In an almost Kenmotsu manifold $M^{2n+1}$, the integral manifolds of $\mathcal{D}$ are totally umbilical submanifolds of $M^{2n+1}$ if and only if $h$ vanishes ". Hence, we can state the following: 
\begin{corollary} \label{cor1.1}
Let $M^{2n+1}$ be an almost Kenmotsu manifold admitting a non-zero HPCV field $V$ such that $\phi V = 0$. Then the integral manifolds of $\mathcal{D}$ are totally umbilical submanifolds of $M^{2n+1}$.
\end{corollary}

\begin{corollary} \label{cor1.2}
If a locally symmetric almost Kenmotsu manifold $M^{2n+1}$ admits a non-zero HPCV field $V$ such that $\phi V = 0$, then $M^{2n+1}$ is a Kenmotsu manifold.
\end{corollary}
The above Corollary follows directly from Theorem 3 of \cite{dp2}.\\

 Proposition 2.1 of \cite{wang2} states that " Any 3-dimensional almost Kenmotsu manifold is Kenmotsu if and only if $h$ vanishes ". Thus we arrive to the following:
 \begin{corollary} \label{cor1.3}
A $3$-dimensional almost Kenmotsu manifold $M^{2n+1}$ admitting a non-zero HPCV field $V$ such that $\phi V = 0$ is a Kenmotsu manifold.
\end{corollary}

\begin{theorem} \label{th1.2}
 Let $M^{2n+1}$ be a complete almost Kenmotsu manifold  admitting a non-zero HPCV field $V$. If $M^{2n+1}$ has positive constant $\xi$-sectional curvature, then either $M^{2n+1}$ is locally a warped product of an almost Kaehler manifold and an open interval or isometric to a sphere.
\end{theorem}
\begin{proof}
If the sectional curvature $K(\xi,X) = c$ of an almost Kenmotsu manifold is a positive constant, then we can easily obtain the following:
\bea
R(\xi,X)\xi = -c[X - \eta(X)\xi]. \label{3.13}
\eea
Now putting $X = \xi$ in (\ref{3.4}) we have
\bea
R(\xi,Y)V = (\xi a)Y - (Ya)\xi + (\xi b)\phi Y + b\phi Y + bhY. \label{3.14}
\eea
Taking inner product of (\ref{3.14}) with $\xi$ we get
\bea
g(R(\xi,Y)V,\xi) = (\xi a)\eta(Y) - (Ya). \label{3.15}
\eea
Again using (\ref{3.13}) we have
\bea
g(R(\xi,Y)V,\xi) = - g(R(\xi,Y)\xi,V) = c[g(Y,V) - \eta(Y)\eta(V)]. \label{3.16}
\eea
Hence from (\ref{3.15}) and (\ref{3.16}) we obtain
\bea
Da - (\xi a)\xi + cV - c\eta(V)\xi = 0. \label{3.17}
\eea
Taking inner product of (\ref{3.14}) with $V$ we get
\bea
(\xi a)V - \eta(V)(Da) - (\xi b)\phi V - b\phi V + bhV = 0. \label{3.18}
\eea
Eliminating $Da$ from (\ref{3.17}) and (\ref{3.18}) we have
\bea
- (\xi a)\phi^2 V - c\eta(V)\phi^2 V - (\xi b)\phi V - b\phi V + bhV = 0. \label{3.19}
\eea
Now differentiating (\ref{3.17}) covariantly along any vector field $X$ and then taking inner product of the resulting equation with $Y$ we infer
\bea
\nonumber && g(\nabla_X Da,Y) - (\xi a)[g(X,Y) - \eta(X)\eta(Y) - g(\phi hX,Y)]  - (X(\xi a))\eta(Y) \\ \nonumber && + c[ag(X,Y) + bg(\phi X,Y)] - c\eta(Y)[g(X,V) - \eta(X)\eta(V) + a\eta(X)] \\ && - c\eta(V)[g(X,Y) - \eta(X)\eta(Y) - g(\phi hX,Y)] = 0. \label{3.20}
\eea
Antisymmetrizing the above equation and using the symmetry of the Hessian operator, that is, $\operatorname{Hess}_a(X,Y) = g(\nabla_X Da,Y) = g(\nabla_Y Da,X)$ we obtain
\bea
\nonumber && (Y(\xi a))\eta(X) - (X(\xi a))\eta(Y) + 2bcg(\phi X,Y) \\ && - c\eta(Y)g(X,V) + c\eta(X)g(Y,V) = 0. \label{3.21}
\eea
Replacing $X$ by $\phi X$ and $Y$ by $\phi Y$ in (\ref{3.21}) we get $2bcg(\phi X,Y) = 0$, which implies $b = 0$ as $c$ is non-zero constant by hypothesis. Then from (\ref{3.18}) we have
\bea
(\xi a)V = (Da)\eta(V). \label{3.22}
\eea
Also from (\ref{3.19}) we obtain
\bea
[(\xi a) + c\eta(V)]\phi^2 V = 0,
\eea
which implies either $\phi^2 V =  0$ or $(\xi a) = - c\eta(V)$.\\
Case 1: If $\phi^2 V =  0$, then we have $V = \eta(V)\xi$ and this implies $\phi V = 0$. Thus from Theorem \ref{th1.1} we infer that $M^{2n+1}$ is locally a warped product of an almost Kaehler manifold and an open interval.\\
Case 2: If $(\xi a) = - c\eta(V)$, then from (\ref{3.22}) we have
\bea
(Da + cV)\eta(V) = 0. \label{3.23}
\eea
Now if $\eta(V) = 0$, then from (\ref{3.22}) we have $\xi a = 0$ as $V$ is non-zero. Hence from (\ref{3.17}) we get $Da = - cV$. Thus in either cases we obtain $Da = - cV$. Differentiating this covariantly along any vector field $X$ and using (\ref{1.2}) we have $\nabla_X Da = - caX$. We are now in a position to apply Obata's theorem \cite{om}: " In order for a complete Riemannian manifold of dimension $n \geq 2$ to admit a non-constant function $\lambda$ with $\nabla_X D\lambda = -c^2 \lambda X$ for any vector $X$, it is necessary and sufficient that the manifold is isometric with a sphere $S^n(c)$ of radius $\frac{1}{c}$ " to conclude that the manifold is isometric  to the sphere $S^{2n+1}(\sqrt{c})$ of radius $\frac{1}{\sqrt{c}}$.
\end{proof}

\section{\textbf{HPCV fields on a class of almost Kenmotsu manifolds}}
In this section, we study HPCV fields on almost Kenmotsu manifolds with $\xi$ belonging to the $(k,\mu)'$-nullity distribution. Let $X\in\mathcal D$ be the eigen vector of $h'$ corresponding to the eigen value $\lambda$. Then from (\ref{2.7}) it is clear that $\lambda^{2} = -(k + 1)$, a constant. Therefore $k \leq -1$ and $\lambda = \pm \sqrt{- k - 1}$. We denote by $[\lambda]'$ and $[-\lambda]'$ the corresponding eigenspaces related to the non-zero eigenvalue $\lambda$ and $-\lambda$ of $h'$, respectively. Before proving our main theorem in this section we recall some results:
\begin{lemma} \label{lem4.1}(Prop. $4.1$ of \cite{dp})
 Let $(M^{2n+1},\phi,\xi,\eta,g)$ be an almost Kenmotsu manifold such that $\xi$ belongs to the $(k,\mu)'$-nullity distribution and $h'\neq0$. Then $k<-1,\;\mu=-2$ and Spec $(h')=\{0,\lambda,-\lambda\}$, with $0$ as simple eigen value and $\lambda=\sqrt{-k-1}$. The distributions $[\xi]\oplus[\lambda]'$ and $[\xi]\oplus[-\lambda]'$ are integrable with totally geodesic leaves. The distributions $[\lambda]'$ and $[-\lambda]'$ are integrable with totally umbilical leaves. Furthermore, the sectional curvature are given by the following:
\begin{itemize}
  \item[(a)]$K(X,\xi)=k-2\lambda$ if $X\in[\lambda]'$ and\\$K(X,\xi)=k+2\lambda$ if $X\in[-\lambda]'$,
  \item[(b)] $K(X,Y)=k-2\lambda$ if $X,Y\in[\lambda]'$;\\$K(X,Y)=k+2\lambda$ if $X,Y\in[-\lambda]'$ and \\$K(X,Y)=-(k+2)$ if $X\in[\lambda]'$, $Y\in[-\lambda]'$,
  \item[(c)] $M^{2n+1}$ has constant negative scalar curvature $r=2n(k-2n).$
\end{itemize}
\end{lemma}

\begin{lemma} \label{lem4.2} (Lemma $4.1$ of \cite{dp})
Let $(M^{2n+1},\phi,\xi,\eta,g)$ be an almost Kenmotsu manifold with  $h'\neq0$ and $\xi$ belongs to the $(k,-2)'$-nullity distribution. Then, for any $X,\; Y \in \chi(M^{2n+1})$,
\bea
(\nabla_X h')Y = - g(h'X + h'^2 X,Y)\xi - \eta(Y)(h'X + h'^2 X) \label{4.1}
\eea
\end{lemma}

\begin{lemma}(Prop. $4.2$ of \cite{dp})\label{lem4.3} Let $(M^{2n + 1},\phi,\xi,\eta,g)$ be an almost Kenmotsu manifold such that $h' \neq 0$ and $\xi$ belonging to the $(k,-2)'$-nullity distribution. Then for any $X_\lambda, Y_\lambda, Z_\lambda \in [\lambda]'$ and $X_{-\lambda}, Y_{-\lambda}, Z_{-\lambda} \in [-\lambda]'$, the Riemann curvature tensor satisfies:
\bea \nonumber	R(X_\lambda,Y_\lambda)Z_{-\lambda} = 0, \eea
\bea \nonumber 	R(X_{-\lambda},Y_{-\lambda})Z_\lambda = 0, \eea
\bea \nonumber 	R(X_\lambda,Y_{-\lambda})Z_\lambda = (k + 2)g(X_\lambda,Z_\lambda)Y_{-\lambda}, \eea
\bea \nonumber 	R(X_\lambda,Y_{-\lambda})Z_{-\lambda} = -(k + 2)g(Y_{-\lambda},Z_{-\lambda})X_\lambda,\eea
 \bea \nonumber 	R(X_\lambda,Y_\lambda)Z_\lambda = (k - 2\lambda)[g(Y_\lambda,Z_\lambda)X_\lambda - g(X_\lambda,Z_\lambda)Y_\lambda], \eea
 \bea \nonumber 	R(X_{-\lambda},Y_{-\lambda})Z_{-\lambda} = (k + 2\lambda)[g(Y_{-\lambda},Z_{-\lambda})X_{-\lambda} - g(X_{-\lambda},Z_{-\lambda})Y_{-\lambda}] .\eea
\end{lemma}

From (\ref{2s1.7}) we have
\bea 
R(X,Y)\xi = k[\eta(Y)X-\eta(X)Y] + \mu[\eta(Y)h'X-\eta(X)h'Y],\label{4.2} 
\eea 
where $k,\;\mu \in \mathbb{R}.$ Also we get from (\ref{4.2})
\bea
 R(\xi,X)Y = k[g(X,Y)\xi-\eta(Y)X]+\mu[g(h'X,Y)\xi-\eta(Y)h'X].\label{4.3}
\eea

 \begin{theorem} \label{th1.3}
A $(k,\mu)'$-almost Kenmotsu manifold with $h' \neq 0$ admitting a HPCV field $V$ such that $\phi V \neq 0$ is either locally isometric to the Riemannian product of an $(n + 1)$-dimensional manifold of constant sectional curvature $-4$ and a flat $n$-dimensional manifold or $V$ is an eigenvector of $h'$.
\end{theorem}
 \begin{proof}
 Substituting $X = \xi$ in (\ref{3.4}) we have
\bea
R(\xi,Y)V = (\xi a)Y - (Ya)\xi + (\xi b)\phi Y + b\phi Y + bhY. \label{4.4}
\eea
Taking inner product of (\ref{4.4}) with $\xi$ we obtain
\bea
g(R(\xi,Y)V,\xi) = (\xi a)\eta(Y) - (Ya). \label{4.5}
\eea
Making use of (\ref{4.2}) we get
\bea
\nonumber g(R(\xi,Y)V,\xi) &=& - g(R(\xi,Y)\xi,V) \\&=& - k\eta(Y)\eta(V) + kg(Y,V) - 2g(h'Y,V). \label{4.6}
\eea
Equations (\ref{4.5}) and (\ref{4.6}) together implies
\bea
- k\eta(Y)\eta(V) + kg(Y,V) - 2g(h'Y,V) = (\xi a)\eta(Y) - (Ya), \label{4.7}
\eea
which implies
\bea
- k\eta(V)\xi + kV - 2h'V = (\xi a)\xi - Da. \label{4.8}
\eea
Now taking inner product of (\ref{4.4}) with $V$ gives
\bea
\nonumber (\xi a)g(Y,V) - (Ya)\eta(V) + (\xi b)g(\phi Y,V) + bg(\phi Y,V) + bg(hY,V) = 0,
\eea
which implies
\bea
(\xi a)V - (Da)\eta(V) - (\xi b)\phi V - b\phi V + bhV = 0. \label{4.9}
\eea
Eliminating $Da$ from (\ref{4.8}) and (\ref{4.9}) we have
\bea
- (\xi a)\phi^2 V - k\eta(V)\phi^2 V - 2\eta(V)h'V - (\xi b)\phi V - b\phi V + bhV = 0. \label{4.10}
\eea
Differentiating (\ref{4.8}) covariantly along any vector field $X$ and using (\ref{1.2}), (\ref{2.3}), Lemma \ref{lem4.2} and the value of $\mu$ from Lemma \ref{lem4.1} we infer
\bea
\nonumber && - k[g(X - \eta(X)\xi - \phi hX,V) + g(\xi,aX + b\phi X)]\xi - k\eta(V)[X - \eta(X)\xi - \phi hX] \\ \nonumber && + k[aX + b\phi X] - 2[- g(h'X + h'^2X,V)\xi - \eta(V)(h'X + h'^2X) + h'(aX + b\phi X)] \\ \nonumber &&  = (\xi a)[X - \eta(X)\xi - \phi hX] + (X(\xi a))\xi - \nabla_X Da.
\eea
Taking inner product of the foregoing equation with $Y$ we obtain
\bea
\nonumber && - k[g(X,V) - \eta(X)\eta(V) - g(\phi hX,V) + a\eta
(X)]\eta(Y) - k\eta(V)[g(X,Y) \\ \nonumber &&- \eta(X)\eta(Y) - g(\phi hX,Y)] + k[ag(X,Y) + bg(\phi X,Y)] - 2[- g(h'X \\ \nonumber && + h'^2X,V)\eta(Y) - \eta(V)g(h'X + h'^2X,Y) + g(ah'X - bhX,Y)] \\ &&  \;\;\;= (\xi a)[g(X,Y) - \eta(X)\eta(Y) - g(\phi hX,Y)] + (X(\xi a))\eta(Y) - g(\nabla_X Da,Y). \label{4.11}
\eea
Antisymmetrizing the above equation and using the symmetry of the Hessian operator, that is, $\operatorname{Hess}_a(X,Y) = g(\nabla_X Da,Y) = g(\nabla_Y Da,X)$ we obtain
\bea
\nonumber &&- k[g(X,V)\eta(Y) - g(Y,V)\eta(X) - g(\phi hX,V)\eta(Y) + g(\phi hY,V)\eta(X)] \\ \nonumber && + 2kbg(\phi X,Y) - 2[- g(h'X + h'^2X,V)\eta(Y) + g(h'Y + h'^2Y,V)\eta(X)] \\ && = (X(\xi a))\eta(Y) - (Y(\xi a))\eta(X). \label{4.12}
\eea
Putting $X = \phi X$ and $Y = \phi Y$ in the previous equation we infer that $2kbg(\phi X,Y) = 0$, which implies $b = 0$ as $k < -1$. Hence from (\ref{4.9}) we have
\bea
(\xi a)V = (Da)\eta(V). \label{4.13}
\eea
Now letting $Y \in [\lambda]'$ in (\ref{4.7}) yields
\bea
\nonumber (k - 2\lambda)g(Y,V) = - (Ya),
\eea
which implies
\bea
Da = (2\lambda - k)V \;\;\; \mathrm{and}\;\;\; (\xi a) = (2\lambda - k)\eta(V). \label{4.14}
\eea
Now using $b = 0$ and the value of $(\xi a)$ from (\ref{4.14}) in (\ref{4.10}) we have
\bea
2(\lambda + 1)\eta(V)(h'V + \phi^2 V) = 0,
\eea
which implies either $\lambda = -1$ or $\eta(V) = 0$ or $h'V = - \phi^2 V$.\\
Case 1: If $\lambda = -1$, then from $\lambda^2 = - k - 1$ we obtain $k = -2$. Now letting $X,\;Y,\;Z \in [\lambda]'$ and noticing that $k = -2$, $\lambda = - 1$, from Lemma \ref{lem4.3} we have
\bea
\nonumber 	R(X_\lambda,Y_\lambda)Z_\lambda = 0,
\eea
and
 \bea
 \nonumber 	R(X_{-\lambda},Y_{-\lambda})Z_{-\lambda} = - 4[g(Y_{-\lambda},Z_{-\lambda})X_{-\lambda} - g(X_{-\lambda},Z_{-\lambda})Y_{-\lambda}],
  \eea
  for any $X_\lambda,Y_\lambda,Z_\lambda \in [\lambda]'$ and $X_{-\lambda},Y_{-\lambda},Z_{-\lambda} \in [-\lambda]'$. Also noticing $\mu = -2$ it follows from Lemma \ref{lem4.1} that $K(X,\xi) = -4$ for any $X \in [-\lambda]' $ and $K(X,\xi) = 0$ for any $X \in [\lambda]' $. Again from Lemma \ref{lem4.1} we see that $K(X,Y) = -4$ for any $X, Y \in [-\lambda]' $ and $K(X,Y) = 0$ for any $X, Y \in [\lambda]' $. As is shown in \cite{dp} that the distribution $[\xi] \oplus [\lambda]'$ is integrable with totally geodesic leaves and the distribution $[-\lambda]'$ is integrable with totally umbilical leaves by $H = -(1 - \lambda)\xi$, where $H$ is the mean curvature tensor field for the leaves of $[-\lambda]'$ immersed in $M^{2n+1}$. Here $\lambda = -1$, then the two orthogonal distributions $[\xi] \oplus [\lambda]'$ and $[-\lambda]'$ are both integrable with totally geodesic leaves immersed in $M^{2n+1}$. Then we can say that $M^{2n+1}$ is locally isometric to $\mathbb{H}^{n+1}(-4)$ $\times$ $\mathbb{R}^n$.\\
Case 2: If $\eta(V) = 0$, then from (\ref{4.14}) we have $(\xi a) = 0$. Then from (\ref{4.8}) we have $Da = 2h'V - kV$. Now equating the value of $Da$ from this and (\ref{4.14}) we get $h'V = \lambda V$. This shows that $V$ is an eigenvector of $h'$.\\
Case 3: If  $h'V = - \phi^2 V = V - \eta(V)\xi$, then applying $h'$ on both side of it we have $h'^2 V = h'V$. Hence using (\ref{2.7}) we obtain $-(k + 2)(V - \eta(V)\xi) = 0$. Now $V - \eta(V)\xi \neq 0$ as $\phi V \neq 0$ by hypothesis. Therefore, we have $k = -2$. Now from $\lambda^2 = - k - 1$, we obtain $\lambda^2 = 1$. Without loss of generality we assume that $\lambda = -1$. Then by the same argument as in Case 1 we get $M^{2n+1}$ is locally isometric to $\mathbb{H}^{n+1}(-4)$ $\times$ $\mathbb{R}^n$. This completes the proof.
\end{proof}

\subsection*{Example}  In \cite{ddey}, the author present an example of a $5$-dimensional $(k,\mu)'$-almost Kenmotsu manifold with $k = - 2$ and $\mu = - 2$. Then by the same argument as in Case 1 of Theorem \ref{th1.3}, $M^{5}$ is locally isometric to $\mathbb{H}^{3}(-4)$ $\times$ $\mathbb{R}^2$.\\
Let $X = \alpha_1 \xi + \alpha_2 e_2 + \alpha_3 e_3 + \alpha_4 e_4 + \alpha_5 e_5$ be any vector field on $M^{5}$ and let $V = e_4$. Then, $\nabla_X V = 0 = a X + b \phi X$, where $a = b = 0$. Hence, $V = e_4$ is an example of a HPCV field, where $\phi V = \phi e_4 = - e_2 \neq 0$. Hence, Theorem \ref{th1.3} is verified.\\

\noindent
\textbf{Acknowledgement:} The author Dibakar Dey is supported by the Council of Scientific and Industrial Research, India (File no: 09/028(1010)/2017-EMR-1) in the form of Senior Research Fellowship.

\end{document}